\documentclass[graybox]{svmult}

\usepackage{amsmath}
\usepackage{amssymb}
\usepackage{mathtools}
\usepackage{graphicx}

\usepackage[T1]{fontenc}
\usepackage[utf8]{inputenc}

\usepackage[english]{babel}
\usepackage{url}

\usepackage{cleveref}

\usepackage{csquotes}

\begin{document}

\title*{Numerical Approximation of the logarithmic Laplacian via sinc-basis}
\author{Patrick Dondl\orcidID{0000-0003-3035-7230} and\\ Ludwig Striet\orcidID{0000-0003-0625-6384}}

\institute{Patrick Dondl \at University of Freiburg, Department for Applied Mathematics, Hermann-Herder-Straße 10 79104 Freiburg, Germany \email{patrick.dondl@mathematik.uni-freiburg.de}
\and Ludwig Striet \at University of Freiburg, Department for Applied Mathematics, Hermann-Herder-Straße 10 79104 Freiburg, Germany \email{ludwig.striet@mathematik.uni-freiburg.de}}

\maketitle

\abstract*{In recent works, the authors of this chapter have shown with co-authors 
	how a basis consisting of dilated and shifted $\text{sinc}$-functions can be used to
	solve fractional partial differential equations. As a model problem, the fractional
	Dirichlet problem with homogeneous exterior value conditions was solved. In this
	work, we briefly recap the algorithms developed there and that -- from a computational
	point of view -- they can be used to solve nonlocal equations given through different
	operators as well. As an example, we numerically solve the Dirichlet problem
	for the logarithmic Laplacian $\log(-\Delta)$ which has the Fourier symbol $\log(\left|\omega\right|^2)$
	and compute its Eigenvalues on disks with different radii in $\mathbb R^2$.
}

\abstract{In recent works, the authors of this chapter have shown with co-authors 
	how a basis consisting of dilated and shifted $\text{sinc}$-functions can be used to
	solve fractional partial differential equations. As a model problem, the fractional
	Dirichlet problem with homogeneous exterior value conditions was solved. In this
	work, we briefly recap the algorithms developed there and that -- from a computational
	point of view -- they can be used to solve nonlocal equations given through different
	operators as well. As an example, we numerically solve the Dirichlet problem
	for the logarithmic Laplacian $\log(-\Delta)$ which has the Fourier symbol $\log(\left|\omega\right|^2)$
	and compute its Eigenvalues on disks with different radii in $\mathbb R^2$.
}

\section{Introduction}
In recent years, nonlocal equations have been a widely studied topic in pure
and applied mathematics. A prototypical example for such an equation is: find
$u$ such that
\begin{equation}
	\label{eq:bvp}
	\begin{aligned}
		{(-\Delta)}^s u &= f\text{ in }\Omega\\
		\quad u &= 0\text{ in }\mathbb R^d\setminus\Omega.
	\end{aligned}
\end{equation}
Above, $f$ is a right-hand side defined on an open, bounded domain $\Omega$. For $s\in(0,1)$, the operator
${(-\Delta)}^s$ is the \emph{integral fractional Laplacian} which can be defined as
\begin{equation*}
	{(-\Delta)}^s u(x) \coloneqq -C(d,s) \text{P.V.}\int_{\mathbb{R}^d} \frac{u(x)-u(y)}{\left|x-y\right|^{d+2s}}\mathop{\mathrm{d}y}.
\end{equation*}
This definition exposes major challenges in the development of numerical
methods for problem of the form of \cref{eq:bvp}: to evaluate the operator
${(-\Delta)}^s u(x)$ at a single point $x\in\mathbb{R}^d$, a singular integral over the whole
$\mathbb{R}^d$ has to be computed. Furthermore, the definition exposes why conditions on
$u$ have to be imposed on $\mathbb{R}^d\setminus\Omega$ instead of only on the boundary
of $\Omega$, thus making them \emph{exterior value conditions.}

Driven by the number of applications of such models, many numerical methods to solve
equations of that kind have been proposed in literature in recent years.

This includes, among others, finite difference methods such as presented in
\cite{Huang2014,Huang2016,Han2022}. In \cite{Valdinoci2009}, the authors derive
the fractional Laplacian as the limit of the long-jump random walk on a grid.
This could also be used to derive the weights $w_k = \left|k\right|^{-(n+\alpha)}$
for a finite difference scheme. Collocation methods have been presented in
\cite{Rosenfeld2019,Burkardt2020,Zhuang2022}. Another large class of methods are
Galerkin methods. Regularity theory and implementation details for a $\mathbb P^1$ finite
element method for \cref{eq:bvp} has been given in \cite{Borthagaray2017}. More aspects
of such methods, including efficient implementations, have been provided in
\cite{Ainsworth2017,Ainsworth2018}. In \cite{Bonito2019}, the authors present
a non-conforming finite element method that amounts in solving a set of elliptic
problems on accordingly truncated domains.  In \cite{Faustmann2023}, the authors show that
exponential convergence can be achieved when using $hp$ finite elements in polygonal
domains. In \cite{Feist2023}, the author show that Duffy transforms can be used
to overcome the difficulties that arise when computing the stiffness matrix for
\cref{eq:bvp} in $d = 3$ dimensions.

In this chapter, we focus on an approach employed and analyzed in \cite{ADS2021,ADS2023}, 
where it has been shown that a basis consisting of dilated and shifted
$\text{sinc}$-functions can be used to approximate nonlocal problems with Dirichlet
exterior value conditions when the operator is given as a Fourier symbol. The 
method presented in these articles can be seen either a collocation method or as 
a Galerkin method. The method relies on the fact that the operator ${(-\Delta)}^s$ --
acting on functions given on $\mathbb{R}^d$ -- can be equivalently defined as a Fourier
multiplier. More precisely, the identity
\begin{equation}
	\label{eq:def_fourier_flap}
	\mathcal F\{{(-\Delta)}^s u\} = \left|\omega\right|^{2s} \left(\mathcal F u\right)(\omega)
\end{equation}
holds. For a proof, see e.g. \cite{Valdinoci2009} or \cite{Kwasnicki2015} where
the equivalence of ten different definitions is shown.

In theory, the definition in \cref{eq:def_fourier_flap} could be used directly
to construct a numerical method via the discrete Fourier transformation as
seen, e.g., in \cite{AntilBartels2017}. However, this introduces an implicit
periodization of the function which may have undesired effects if homogeneous
exterior value conditions as in \cref{eq:bvp} are desired.

The aforementioned approach presented in \cite{ADS2021,ADS2023} differs from that
as $\text{sinc}$-functions are used as a basis for the numerical approximations. Those functions
combine (i) a reasonable decay in the physical space that can appropriately model the exterior value
conditions and (ii) a simple Fourier transformation which is essentially the indicator
function of a box. While the focus of \cite{ADS2021} is the efficient implementation
of the method along with numerical experiments, the focus of \cite{ADS2023} are
numerical analysis and a convergence proof for the discretization of \cite{ADS2023}
under mild assumptions on $\Omega$ and $f$. We summarize the details that are important
for this article in \cref{sec:sinc_method}.

In this note, we show that -- from a computational point of view -- the
developed $\text{sinc}$-function based techniques can be used to approximate a much broader range of problems.
More specifically, we can compute an operator $L$ as long it is defined through
a Fourier symbol $m(\omega)$ via
\begin{equation}
	\label{eq:def_fourier_operator}
	\mathcal F\{L u\} = m(\omega)\left(\mathcal F u\right)(\omega)
\end{equation}
and solve associated (nonlocal) partial differential equations in the sense of collocation
methods.

As an example, we numerically study  some of the properties of the \emph{logarithmic Laplacian}
$\log(-\Delta)$ which can be defined via the identity
\begin{equation*}
	\mathcal F\{\log(-\Delta) u\} = \log(\left|\omega\right|^2)(\mathcal F u)(\omega).
\end{equation*}
The logarithmic Laplacian formally arises as the derivative $\partial_s\Big|_{s=0}{(-\Delta)}^s$
of the fractional Laplacian at $s = 0$ \cite{ChenWeth2019}. This operator has been
subject to investigations in articles in recent years. In \cite{ChenWeth2019},
the authors study the Dirichlet problem for the logarithmic Laplacian and
the spectral properties of the logarithmic Laplacian are subject of, among others, \cite{LaptevWeth2021}
and bounds for the eigenvalues are given in \cite{Chen2022}.

Recently, the numerical approximation of the logarithmic Laplacian on an interval
has been studied in \cite{HernndezSantamara2025}. In this work, the authors
show the implementation and analysis of a finite element method and they establish
error estimates in appropriately defined weighted function spaces. The analysis is
substantiated by numerical experiments. Furthermore, they show that the eigenvalues
of the discretized stiffness matrix converge to the  eigenvalues of the
logarithmic Laplacian as the spatial discretization parameter $h$ approaches $0$.

In this work, we present an approach suitable for dimensions higher than one, obtaining results that are in good alignment with theoretical predictions. A rigorous numerical analysis of our method is an open problem and ongoing work.

\section{Mathematical Preliminaries}
We introduce some of the basic facts regarding the Fourier transform that are required in the following sections.
For details, we refer to the many textbooks on the topic such as, e.g., \cite{Grafakos2010}.

In this article, we use the conventions 
\begin{equation*}
	\mathcal F u(\omega) = (2\pi)^{-d} \int_{\mathbb{R}^d} u(x) \mathrm{e}^{-i\omega x}\mathop{\mathrm{d}x}, \qquad
	\mathcal F^{-1} \hat u(x) = \int_{\mathbb{R}^d} \hat u(\omega) \mathrm{e}^{\mathrm i\omega x}\mathop{\mathrm d\omega}.
\end{equation*}
for the Fourier transform on $L^2(\mathbb{R}^d)$ and its inverse. The
definition is considered to be extended to the space of tempered distributions accordingly
where needed. An important tool in Fourier analysis that we make use of is the 
Fourier scaling theorem which states that the identity
\begin{equation*}
	\mathcal F(u(x/h))(\omega) = \left|h\right|^{d} \mathcal F u(h\omega)
\end{equation*}
holds for $h\in\mathbb R$.

Initially, our algorithms are implemented so that they work on a domain $\Omega\subset(0,1)^d$. By translational invariance, we can easily see that no change is required to consider $\Omega\subset(-1/2,1/2)^d$. Arbitrarily large domains can then be treated by scaling the Fourier symbols, as justified by the following lemma.

\begin{lemma}
	\label{lemma:domain_scaling}
	Let $L$ a linear operator with symbol $m(\omega)$ and $L_{r/R}$ the operator
	with symbol $m(r/R\Omega)$. Let $u$ solve
	\begin{equation*}
		\begin{aligned}
			L_{r/R} u &= \lambda u \text{ in }B_r\\
			\quad u &= 0\text{ in }\mathbb R^d\setminus B_r
		\end{aligned}
	\end{equation*}
	and $v(x) = u(r/R x)$. Then $v$ solves
	\begin{subequations}
		\begin{align}
			L v &= \lambda v \text{ in }B_R\tag{a}\label{eq:Lv_lambdav_a}\\
			\quad v &= 0\text{ in }\mathbb R^d\setminus B_R\tag{b}\label{eq:Lv_lambdav_b}
		\end{align}
	\end{subequations}
\end{lemma}
\begin{proof}
	The second part (\ref{eq:Lv_lambdav_b}) is obvious. The first part (\ref{eq:Lv_lambdav_a})
	is shown by a simple computation: let $x\in B_R$, then $r/R x\in B_r$ and
	\begin{align*}
		\lambda v(x) &= \lambda u(r/R x) = L_{r/R} u(r/R x) \\
								 &= \mathcal F^{-1}\{ m(r/R\omega) \hat u(\omega) \}(r/R x) \\
								 &= \int_{\mathbb{R}^d} m(r/R\omega) \hat u(\omega) \mathrm{e}^{\mathrm i r/R  x \cdot\omega} \mathop{\mathrm d\omega} \\
								 &= \int_{\mathbb{R}^d} m(\omega) \hat u(R/r\omega) \mathrm{e}^{\mathrm i x \omega } (r/R)^{-d}\mathop{\mathrm d\omega}  \\
								 &= \int_{\mathbb{R}^d} m(\omega) \hat v(\omega) \mathrm{e}^{\mathrm i x \omega}\mathop{\mathrm d\omega} \\
								 &= L v(x)
	\end{align*}
	where we used that $\hat v(\omega) = (R/r)^d \hat u(R/r\omega)$ as a consequence
	of the Fourier scaling theorem.
\end{proof}
Similarly, we obtain the following lemma.
\begin{lemma}
	\label{lemma:rhs_scaling}
	Let $L$ a linear operator with symbol $m(\omega)$ and $L_{r/R}$ the operator
	with symbol $M(r/R\Omega)$. Let $u$ solve
	\begin{equation*}
		\begin{aligned}
			L_{r/R} u &= f \text{ in }B_r\\
			\quad u &= 0\text{ in }\mathbb R^d\setminus B_r
		\end{aligned}
	\end{equation*}
	and $v(x) = u(r/R x)$. Then $v$ solves
	\begin{equation*}
		\begin{aligned}
			L v &= f(r/R x) \text{ in }B_R\\
			\quad v &= 0\text{ in }\mathbb R^d\setminus B_R
		\end{aligned}.
	\end{equation*}
\end{lemma}
\begin{proof}
	Same as the proof of \cref{lemma:domain_scaling}, starting with $f(r/Rx) =
	L_{r/R}u(r/R x)$ for $x \in B_R$.
\end{proof}

\section{The sinc-method for nonlocal operators}
\label{sec:sinc_method}
In this section, we briefly recap the details needed for the implementation
of the method presented in \cite{ADS2021} to solve nonlocal equations. 

The $\text{sinc}$-function is defined as
\begin{equation}
	\label{eq:def_sinc}
	\text{sinc}(x) \coloneqq \frac{\sin(\pi x)}{\pi x}
	= \int_{[-\pi,\pi]} \frac{1}{2\pi} \mathrm{e}^{\mathrm i\omega \cdot x}\mathop{\mathrm d\omega}
	= \mathcal F^{-1} \left\{ 1	2\pi \chi_{-\pi,\pi} \right\}
\end{equation}
where
\[
	\chi_{[-\pi,\pi]}(\omega) = \begin{cases}
		1&\text{if }\omega\in[-\pi,\pi] \\
		0&\text{otherwise}
	\end{cases}.
\]
We notice from \cref{eq:def_sinc} that the $\text{sinc}$ function is obtained as the inverse Fourier
transform of the indicator function of the interval in $\mathbb R$.

To approximate problems in the form of \cref{eq:bvp} and, more general, operators of
the form of \cref{eq:def_fourier_operator}, we use dilated and shifted tensor
products of the $\text{sinc}$-function defined in \cref{eq:def_sinc}. Namely,
for a positive integer $N$ and a multiindex $k = (k_0,\ldots,k_d)\in\{0,\ldots,N-1\}^d$,
we define the function 
\[
	\varphi^N_k(x) = \prod\limits_{i=1}^d \varphi(Nx - k_i).
\]
For a multiindex $k\in\mathbb Z^d$, we define the grid points $x_k = k/N$.
Note that the basis functions fulfill the property
\[
	\varphi^N_k(x_j) = \delta_{k,j} = \begin{cases}
		1&\text{if }k=j \\
		0&\text{otherwise}
	\end{cases}.
\]
We define the discrete, finite dimensional function space
\[
	\mathbb V_h(\Omega) = \Big\{ v_h(x) = \sum\limits_{k\in\mathbb Z^d} v_k\varphi^N_k(x)\,|\, v_k\in\mathbb R, v_k = 0 \text{ if }x_k\not\in \Omega\Big\}
\]
and solve the discret equation: find $u_h\in\mathbb V_h(\Omega)$ that fulfills
\begin{equation}
	\label{eq:discrete_eq_bvp}
	L u_h(x_k) = f(x_k).
\end{equation}
To solve \cref{eq:discrete_eq_bvp}, we have to discretize the operator $L$. The details
are provided in great detail in \cite{ADS2021,Striet2024}. Briefly, the idea is the following.
Evaluating $L v_h(x_\kappa)$ for $v_h\in V_h(\Omega)$ and a grid point $x_\kappa$
results in computing
\begin{align*}
	{(-\Delta)}^s v_h(x_\kappa) &= L \Big(\sum\limits_{k\in\Omega_h}v_k \varphi^N_k(x_\kappa)\Big) \nonumber\\
			&= \sum\limits_{k\in\Omega_h}v_k \underbrace{\left(L\varphi^N_k\right)(x_\kappa)}
			_{\eqqcolon\Phi^N(\kappa-k)}\nonumber \\
			&= \sum\limits_{k\in{\mathcal I}_N^d}v_k \Phi^N(\kappa-k).
\end{align*}
with ${\mathcal I}_N = \{0,\ldots,N-1\}$. This is a discrete convolution and
can be evaluated efficiently using the discrete fourier transform once $\Phi^N(\kappa -k)$ is known
for all $\kappa - k$. The authors show in the aforementioned works that while evaluating
$\Phi^N_k = L \varphi^N_k$ directly is hard, computing its discrete Fourier
transform can be done as follows: Using the fact that
\[
	\Phi^N(\kappa - k) = L \varphi^N_k(x_\kappa) = L\varphi^N_k(x_\kappa - x_k)
		= \mathcal F^{-1} \left( m(\omega)\hat\varphi^N(\omega) \right)(x_\kappa - x_k)
\]
and the definition of the discrete Fourier transform of size $(2N)^{d}$, we know
for the $k$-th coefficient of the discrete Fourier transform of $\Phi$ that
it can be obtained through the equation
\[
	\widehat\Phi^N_k =
	(2\pi)^{-d}\int\limits_{[-N,N]^d} m(\pi\omega) Y_d\left( \frac{\pi}{N} (\omega - k) \right)\mathop{\mathrm d\omega}
\]
where, for $\omega = (\omega_1,\ldots,\omega_d)^T$,
\[
	Y_d(\omega) = \prod\limits_{i=1}^d Y(\omega_i),\quad
	Y(x)\coloneqq\sum\limits_{j=-N}^{N-1} \mathrm{e}^{\mathrm i j x} = \begin{cases}
		\frac{\mathrm{e}^{-\mathrm i Nx}(\mathrm{e}^{2\mathrm i Nx} - 1)}{\mathrm{e}^{\mathrm i x}-1}&\text{if } \mathrm{e}^{\mathrm i x} -1 \neq 0 \\
		2N&\text{otherwise}
	\end{cases}.
\]
This is still an osciallating integral, but it can be computed approximately as follows.
First, we split the integral up into $(2N)^d$-many integrals over cubes with side length
$1$ and obtain
\begin{equation}
	\label{eq:PHI_hat_sum_of_integrals}
	\widehat\Phi^N_k = \sum\limits_{j\in\I'^d_{2N}} \int_{Q_j} 
	m(\pi\omega) Y_d\left( \frac{\pi}{N} (\omega - k) \right)\mathop{\mathrm d\omega}
\end{equation}
where $\mathcal I_{2N}' \coloneqq \{-N,\ldots N-1\}$ and
$Q_j \coloneqq [j_1, j_1 + 1] \times\cdots \times [j_d, j_d + 1]$.
To evaluate this, we choose a quadrature rule $(x_i,
\alpha_i)_{i=1,\ldots,N_Q}$ is a quadrature rule on $[0,1]^d$, apply it on each of
the cubes and compute the values of $\widehat\Phi^N_k$ via the formula
\begin{equation}
	\label{eq:PHI_hat_Nk}
	\widehat{\Phi}^N_{k} \approx
	(2N)^{-d}\sum\limits_{i=1}^{N_Q}\alpha_i\sum\limits_{j\in{\mathcal I}'^d_{2N}}m(\pi(j+x_i))
		Y_d\left( - \frac{\pi}{N} \left(k-j\right) + \frac{\pi}{N}x_i\right)
\end{equation}
with ${\mathcal I}'_{2N} = \{-N,\ldots,N-1\}$. Here, each of the inner sums has
the structure of a discrete convolution again,
which can be used to implement \cref{eq:PHI_hat_Nk} efficiently. This is still 
computationally demanding, but has to be done only once for each $N$ and each operator.

To implement the above approach, one has to choose a quadrature rule. In principle, there is no particular restriction, except that the same rule has to be used on each cube.
Tensor Gauß-Legendre rules have been used for the relevant implementations in \cite{ADS2021} and it
has been seen in numerical experiments that this is a legitimate
choice for the fractional Laplacian, even though the symbol exhibits reduced
regularity near the origin, making evaluation of \cref{eq:PHI_hat_Nk} more difficult for multi-indices $j$ where one or more components are $0$ or $-1$.

In the case of the logarithmic Laplacian, the expression in
\cref{eq:PHI_hat_Nk} has an even more pronounced singularity near the origin
which thus should be treated properly. This can be achieved for example by
means of a Duffy transform, which in our case consists of a singular domain
transformation to cancel a singularity of an integrand at a corner of the
domain.

We illustrate the procedure by computing the integral
\[
	\int_{Q_0} m(\pi\omega) Y_d\left( \frac{\pi}{N} \omega \right)\mathop{\mathrm d\omega} 
		= \int_0^1\int_0^1 m(\pi\omega_1, \pi\omega_2) \, y_0(\omega_1, \omega_2)\mathrm d\omega_2
		\mathrm d\omega_1
\]
where we abbreviate
\[
	y_0(\omega_1, \omega_2) = Y_d\left( \frac{\pi}{N} \omega \right).
\]
This is precisely the integral over the $j$th cube, $j = (0,0)^T$, for $k = (0,0)^T$
in \cref{eq:PHI_hat_sum_of_integrals}. It is clear that the integrand has the aforementioned 
singularity at $(\omega_1,\omega_2) = (0,0)$ for $m(\omega_1,\omega_2) = \log(\omega_1^2, \omega_2^2)$.
The idea of the Duffy transform is to split the integral over the cube into two integrals
over triangles via
\begin{align}
	\label{eq:int_triangle_split}
	&\int_0^1\int_0^1 m(\pi\omega_1, \pi\omega_2) \, y_0(\omega_1, \omega_2)\mathrm d\omega_2
		\mathrm d\omega_1 \nonumber\\
		=& \int_0^1\int_0^{\omega_1} m(\pi\omega_1, \pi\omega_2) \, y_0(\omega_1, \omega_2)\mathrm d\omega_2
		\mathrm d\omega_1 \\
		 &\quad + \int_0^1\int_0^{\omega_2} m(\pi\omega_1, \pi\omega_2) \, y_0(\omega_1, \omega_2)\mathrm d\omega_1
		\mathrm d\omega_2.
\end{align}
so that the singularity is in the corner of one of them. Then, the singularity
in the first integral is removed by transforming it to an integral over a cube again
as
\begin{align*}
	&\int_0^1\int_0^{\omega_1} m(\pi\omega_1, \pi\omega_2) \, y_0(\omega_1, \omega_2)\mathrm d\omega_2
		\mathrm d\omega_1 \\
		=& \int_0^1\int_0^1 m(\pi\omega_1, \pi\omega_1\eta) \, y_0(\omega_1, \omega_1\eta) 
			\omega_1\mathrm d \omega_1 \mathrm d\eta.
\end{align*}
The second integral is transformed the same way as
\begin{align*}
	&\int_0^1\int_0^{\omega_2} m(\pi\omega_1, \pi\omega_2) \, y_0(\omega_1, \omega)\mathrm d\omega_1
		\mathrm d\omega_2 \\
	=& \int_0^1\int_0^{1} m(\pi\omega_2\eta, \pi\omega_2) \, y_0(\omega_2\eta, \omega)\omega_2\mathrm d\eta
		\mathrm d\omega_2
\end{align*}
and we obtain 
\begin{align}
	\label{eq:int_duffy_transformed}
	&\int_0^1\int_0^1 m(\pi\omega_1, \pi\omega_2) \, y_0(\omega_1, \omega)\mathrm d\omega_2
		\mathrm d\omega_1 \nonumber\\
		=& \int_0^1\int_0^1 m(\pi\omega_1, \pi\omega_1\eta) \, y_0(\omega_1, \omega_1\eta) 
			\omega_1\mathrm d \omega_1 \mathrm d\eta \nonumber\\
		 &\qquad+ \int_0^1\int_0^{1} m(\pi\omega_2\eta, \pi\omega_2) \, y_0(\omega_2\eta, \omega)\omega_2\mathrm d\eta
		\mathrm d\omega_2 \nonumber\\
		=& \int_0^1\int_0^1 \big(
				m(\pi\omega, \pi\eta\omega)y(\omega,\eta \omega) + m(\pi\eta\omega, \pi\omega)y(\eta\omega,\omega)
		\big)\omega \mathrm d\omega \mathrm d\eta.
\end{align}
by plugging both into \cref{eq:int_triangle_split}, renaming $\omega_1 = \omega$
in the first integral, $\omega_2 = \omega$ and changing the order of integration in the 
second integral and summarizing both into one integral over $(0,1)^2$.

As $Y(\cdot)$ is smooth and bounded, the singularity results only from $m(\omega) = \log(\left|\omega\right|^2)$.
The transform as described in \cref{eq:int_duffy_transformed} transforms this singularity
into an expression of the form
\[
	\omega m(\pi\omega, \pi\omega\eta) = \omega\log( (\pi\omega)^2 + (\pi\omega\eta)^2 )
		= \omega \log(\pi^2 (1+\eta^2) \omega^2)
\]
which remains bounded and can be integrate using standard quadrature formula.

However, as this affects only the cubes adjacent to the origin, we only want to
apply an appropriate rule on these cubes, but we have to apply the same rule to
all cubes in order to compute the coefficients via \cref{eq:PHI_hat_Nk} which
is necessary for an efficient algorithm.

We solve this issue by first computing $\widehat\Phi^N_k$ for all $k$ via \cref{eq:PHI_hat_Nk}.
Then, we subtract the summands that belong to cubes near the origin from each of
the coefficients and compute those again using a quadrature rule on the integral
transformed to the form of \cref{eq:int_duffy_transformed}.

\section{Numerical Experiments}
In this section, we show the results of some new numerical experiments. This
includes numerical error rates for the Dirichlet problem with integral fractional
Laplacian in $d = 4$ spatial dimensions in \cref{sub:numerical_flap_4d},
and experiments regarding the eigenvalues of the logarithmic Laplacian
on the ball in \cref{sub:numerical_evals_loglap} and the Dirichlet problem for
the logarithmic Laplacian on the ball in \cref{sub:numerical_dirichlet_loglap}.

\subsection{The Dirichlet problem for the fractional Laplacian on the ball}
\label{sub:numerical_flap_4d}
A standard problem to test numerical methods for the integral fractional
Laplacian is to find $u$ that fulfills
\begin{equation}
	\label{eq:frac_torsion}
	\begin{aligned}
		{(-\Delta)}^s u &= 1\text{ in }B_1\\
		\quad u &= 0\text{ in }\mathbb R^d\setminus B_1
	\end{aligned}
\end{equation}
which has the analytic solution
\[
	u = C_u(d,s)  \max\left(0, (1 - |x|^2)\right)^s,
		\quad C_u(d,s) = \frac{\Gamma\left( d/2\right)}{2^{2s} \Gamma\left( d/2+s \right)\Gamma(1+s)},
\]
see e.g. \cite{Bucur2016} also for the exact value of $C_u(d,s)$. Error decay
rates for the $\text{sinc}$-method for this problem in the energy norm and in the
$L^2$ norm have been computed numerically in \cite{ADS2021} in $d = 2$
and $d = 3$ spatial dimensions and proven later in \cite{ADS2023} for 
arbitrary spatial dimensions. Although the proofs are only for the error decay rate
in the energy norm, the numerically computed decay rates in $L^2$ match
well with what has been proven for finite element methods and are $\sim h^{\min(1/2+s, 1)}$.

In principle, the method can be implemented in arbitrary spatial dimensions, although
the \enquote{curse of dimensionality} is a problem in higher dimensions. Still,
we show numeric error decay rates in the $L^2$-norm in $d = 4$ with $h$ up to $2^{-7}$
which corresponds to approximately $54.3\cdot 10^6$ unknowns in $\Omega$.

We show the errors that we compute in \cref{fig:errors_4d}. The results
suggest that we obtain the same error decay rates as in $d = 2$ and $d = 3$ spatial
dimensions.

\begin{figure}
	\resizebox{\textwidth}{!}{\input{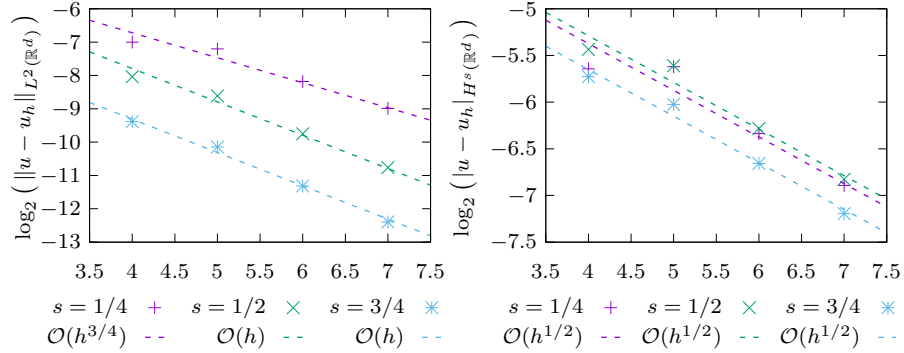}}
	\caption{The numerical errors that we compute for the problem in \cref{eq:frac_torsion} 
		in $d = 4$ spatial dimensions in the $H^s$ seminorm (left) and in the $L^2$ norm (right).
		The dashed lines indicate that we obtain the same error decay rates as in $d = 2$ and $d = 3$.}
	\label{fig:errors_4d}
\end{figure}

\subsection{The eigenvalues of the logarithmic Laplacian on the ball}
\label{sub:numerical_evals_loglap}
As has been noted in \cite{LaptevWeth2021}, the eigenvalues of the logarithmic Laplacian
with Dirichlet boundary conditions on the unit ball $B_R$ depend on the radius $R$.
Furthermore, for specific radii $R_\ell$, zero is an eigenvalue of the logarithmic
Laplacian. The eigenvalues of the logarithmic Laplacian on the ball $B_R$ 
are then
\begin{equation}
	\label{eq:eval_scaling}
	\lambda^{(\ell)} = 2\log(R/R_\ell),
\end{equation}
see \cite[Lemma 2.5]{LaptevWeth2021}, \cite{HernndezSantamara2025}.

In a first numerical experiment, we aim to approximate compute the smallest eigenvalues of the
approximation of the Dirichlet logarithmic Laplacian on $B_R$ for different values of $R > 0$.
Once known, these can be used to simplify implementations, we seek the eigenvalues of the operator with
symbol $\log\big(\left|r/R\omega\right|^2\big)$ on $B_r \subset (0,1)^2$ to obtain the eigenvalues
of the (approximated) logarithmic Laplacian with symbol $2\log(\left|\omega\right|)$ on $B_R \subset \mathbb{R}^d$ as
justified by \cref{lemma:domain_scaling}.

To compute the eigenvalues $\lambda_h^{(\ell)}$, we setup the operator with $N = 2^{10}$ grid points in
each spatial direction and a $7\times 7$-point tensor Gauß-Legendre rule. We use
the Spectra Library \cite{Spectra} to compute the eigenvalues of $\Phi^N$ then. 

We remark that we are currently unable to provide a proof that the computed eigenvalues indeed
converge to the exact eigenvalues of the logarithmic Laplacian. However, it is worth
to note that the pairs of eigenvalues $\lambda_h^{(\ell)}$ and eigenvectors $\vec v_h^{(\ell)}$
that fulfill $(\Phi^N\vec v^{(\ell)})_k = (\lambda_h\vec v^{(\ell)})_k$ for $k\in\Omega_N$
do represent eigenvalues and eigenfunctions of the logarithmic Laplacian in
the following sense: Let $v^{(\ell)}_h$ the $\text{sinc}$-function associated with
the the coefficient vectors $\vec v^{(\ell)}$, then
\[
	\log(-\Delta) v_h(x_k) = (\Phi^N \vec v^{(\ell)})_k = (\lambda^{(\ell)}_h\vec v^{(\ell)})_k
	= \lambda_h^{(\ell)} v_h^{(\ell)}(x_k)\quad \text{ for }x_k\in B_R.
\]
We show the eigenvalues that we compute for different values of $R$ in \cref{fig:R_evals}, 
where we also see that the eigenvalues scale logarithmically as expected.

\begin{figure}
	\begin{center}
		\input{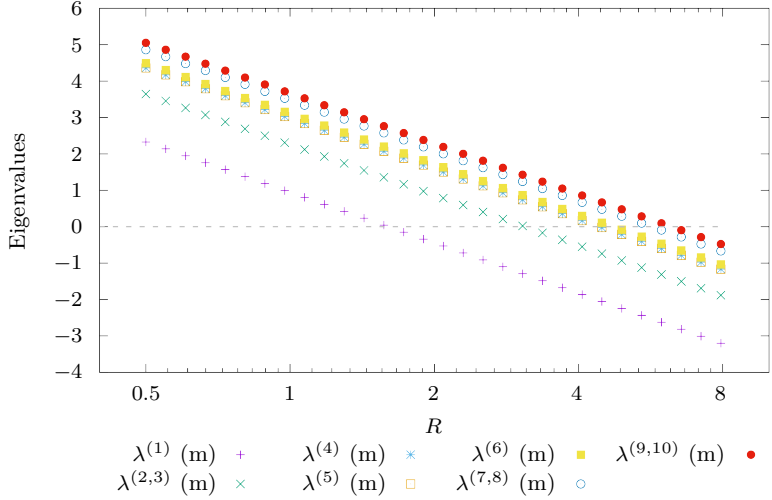}
	\end{center}
	\caption{The first Dirichlet eigenvalues numerically computed on $B_R\subset\mathbb R^2$
		using the $\text{sinc}$ method. We notice that the eigenvalues scale logarithmically
		as expected.}
		\label{fig:R_evals}
\end{figure}

In a next numerical experiment, we aim to seek the exact values of $R$ when $0$
is an eigenvalue of the logarithmic Laplacian on $B_R$ with Dirichlet exterior 
value conditions. These values could be used to estimate the Eigenvalues on balls
with arbitrary radii via \cref{eq:eval_scaling}.

To find the radius for which the $\ell$th eigenvalue is $0$, we take (i) the
biggest $R$ for which $\lambda^{(\ell)}$ is greater than $0$ and (ii) the
smallest $R$ for which $\lambda^{(\ell)}$ is smaller than $0$ from \cref{fig:R_evals}
and iteratively refine them until we find $R_\ell$ so that $\lambda^{(\ell)} =
0$. We show the values of $R_\ell$ in \cref{tab:ell_R_ell}.
Furthermore, we show some of the eigenvectors that belong to the respective
$0$-eigenvalue in \cref{fig:eigenvalues_lambda0}.

\begin{table}
	\caption{The values of $R_\ell$ for which the $\ell$th eigenvalue is $0$, rounded
	to 4 decimal places.}
	\label{tab:ell_R_ell}
	\small
	\begin{center}
	\begin{tabular}{r|c|c|c|c|c|c|c|c|c|c|c|}
		$\ell$   & 1      & 2      & 3      & 4      & 5      & 6      & 7      & 8      & 9      & 10 \\
		$R_\ell$ & 1.6015 & 3.0910 & 3.0910 & 4.4221 & 4.4251 & 4.7248 & 5.6846 & 5.6846 & 6.2476 & 6.2476
	\end{tabular}
	\end{center}
\end{table}

\begin{figure}
	\begin{center}
		\resizebox{\textwidth}{!}{\input{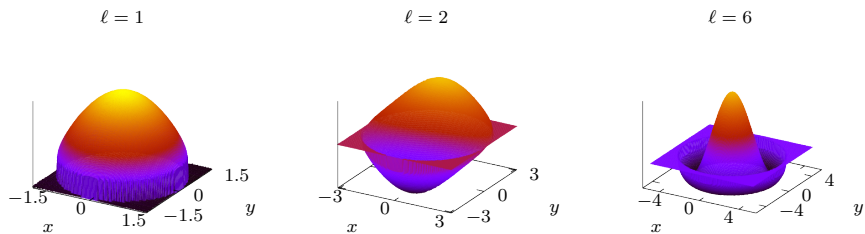}}
	\end{center}
	\caption{The eigenvectors that belong to the eigenvalue $\lambda_\ell = 0$ for
		the radii $R_\ell$ in \cref{tab:ell_R_ell}.
	}
	\label{fig:eigenvalues_lambda0}

\end{figure}

\subsection{The Dirichlet Problem for the logarithmic Laplacian}
\label{sub:numerical_dirichlet_loglap}
In a second experiment, we show numeric solutions to the Dirichlet problem
for the logarithmic Laplacian on the Ball with radius $R$ in $\mathbb{R}^d$. More precisely,
we aim to find numeric solutions to the problem 
\begin{equation*}
	\begin{aligned}
		\log(-\Delta) u_h(x_k) &= f(x_k)\text{ if }x_k \in \Omega\\
		u_h(x_k) &= 0\text{ if }x_k\not\in\Omega
	\end{aligned}
\end{equation*}
where $\Omega\subset(0,1)^d$ and we use the already described scaling procedure
to transform the problem to a problem on $B_R$, see \cref{lemma:rhs_scaling}.
We choose $f = 1$ in $\Omega$
and solve the problem exemplarily for $\Omega = B_R$, $R = 4$ and $R = 6$. We
show the solutions we obtain numerically in \cref{fig:solutions_loglap_dirichlet}.

As a numerical analysis for our method is still pending, we perform a numerical
error analysis. To do so, we compute a solution $u_{h_{\min}}$ at a fine
spatial resolution $h_{\min} = 2^{-12}$ and compare this solution to solutions
$u_h$ computed at coarser solutions. We approximate the $L^2$-error as
\[
	\left\|u_h - u_{h_{\min}}\right\|_{L^2(\mathbb R^d)} \approx
		\Big(\frac{1}{h^2} \sum\limits_{k\in\mathbb Z^d} (u_h(x_k) - u_{h_{\min}}(x_k))^2\Big)^{1/2}.
\]
The error decays with a rate of approximately $h^{-1}$ in this experimental
analysis as shown in \cref{fig:loglap_numeric_error_ana}

\begin{figure}
\begin{center}
	\resizebox{\textwidth}{!}{\input{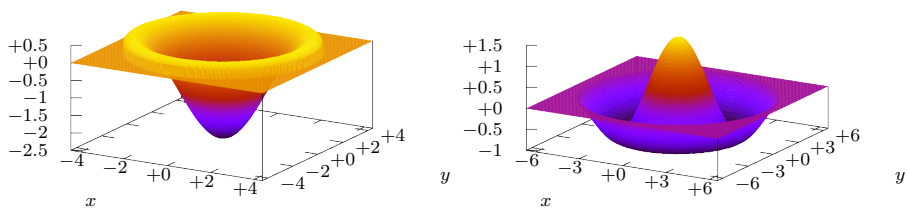}}
\end{center}
\caption{The solutions we obtain to the Dirichlet problem for the logarithmic Laplacian
on balls with radius $R = 4$ (left) and $R = 6$ (right). As expected, the solutions
oscillate for larger values of $R$.}
\label{fig:solutions_loglap_dirichlet}
\end{figure}

\begin{figure}
\begin{center}
	\input{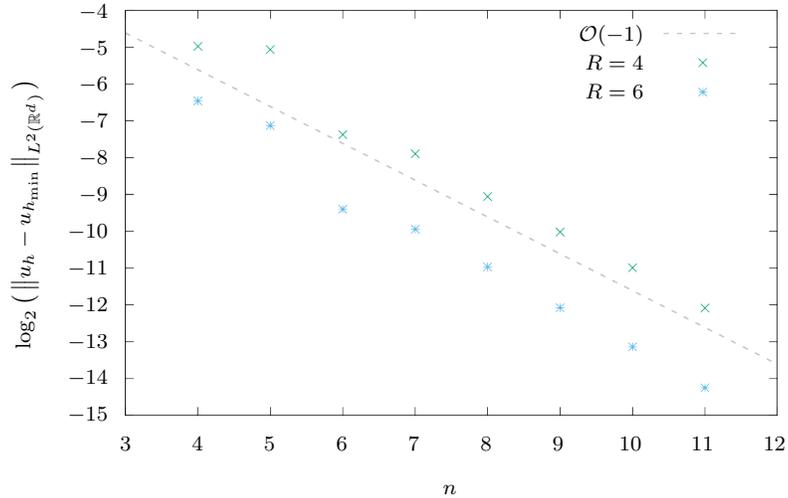}
\end{center}
\caption{The approximately computed
	$\left\|\,\cdot\,\right\|_{L^2(\mathbb R^d)}$
	errors where $u_h$ is the numerical solution
	obtained with spatial resolution $h = 2^{-n}$ and $u_{h_{\min}}$ is the numerical
	solution obtained at the minimal spatial resolutoin $h_{\min} = 2^{-12}$.
	For this experiment, the error rate is approximately $-1$ as indicated by the
	dashed grey line.}
\label{fig:loglap_numeric_error_ana}
\end{figure}

\bibliographystyle{spmpsci}
\bibliography{bibliography}

\end{document}